\documentclass[reqno,12pt]{amsart}
\usepackage[utf8]{inputenc}
\usepackage[left=1.2in, right=1.2in, top=1in]{geometry}

\usepackage{amsmath, amssymb,amsthm, mathrsfs,color,times,textcomp,verbatim}
\usepackage{xcolor}
\usepackage[colorlinks=true]{hyperref}
\hypersetup{urlcolor=blue, citecolor=red, linkcolor=blue}
\usepackage[square,numbers]{natbib}
\usepackage{pgfplots}
\usepackage{tikz}
\numberwithin{equation}{section}
\theoremstyle{plain}
\newtheorem{theorem}{Theorem}[section]
\newtheorem{proposition}[theorem]{Proposition}

\newtheorem{lemma}[theorem]{Lemma}

\newtheorem{corollary}[theorem]{Corollary}

\allowdisplaybreaks[4]
\usepackage{graphicx}

\theoremstyle{definition}

\usepackage{amssymb,amsthm,amsmath}
\title[Classification on 15 dimensional octonionic Heisenberg group]{Classification of finite energy positive solutions to Yamabe-type equation on the fifteen dimensional octonionic Heisenberg  group}

\author[]{Daowen Lin}
\address{ Geometric Partial Differential Equations Unit, Okinawa Institute of Science and
Technology Graduate University, Okinawa 904-0495, Japan}
\email{daowen.lin@oist.jp}
\begin{document}
\begin{abstract}
  We classify finite-energy positive solutions to the Yamabe-type equation on the 15-dimensional octonionic Heisenberg group, whose algebra is non-associative. Extremals for the Folland-Stein-Sobolev inequality on this group are explicitly described. This confirms the conjecture of Garofalo and Vassilev [Duke Math. J. 2001] on the $15$ dimensional octonionic Heisenberg group.
\end{abstract}
\maketitle

\begin{center}


\bigskip


\end{center}

\section{Introduction}

We consider  Yamabe type equation on the $15$ dimensional octonionic Heisenberg group $G$. 
Let $u(\xi)=u(x,t)$, where $(x,t)\in \mathbb{R}^{8}\times \mathbb{R}^{7}$, be a solution of the following equation 
\begin{equation}\label{1.1}
    \begin{aligned}
        &\Delta_G u+160u^\frac{Q+2}{Q-2}=0\\& u\in \mathscr{D}^{1,2}(G),\,\,u> 0   
    \end{aligned}
\end{equation}
where $Q=22$ is the homogeneous dimension of $G$. 

Equation (\ref{1.1}) has been studied due to its  connection with  Sobolev type inequality (\ref{1.3}) and   Yamabe type problem.

For the Heisenberg group case, which is related to the CR Yamabe problem: \textit{Given a compact, strict pseudo-convex CR manifold, find a choice of contact form for which the Webster-Tanaka pseudo-hermitian scalar curvature is constant}, Jerison and Lee \cite{JL1988} classified all finite energy positive solutions to  Yamabe type equation by introducing three family of differential identities. For further discussions,  see \cite{G2001a} \cite{G2001b} \cite{JL1987} \cite{JL1989} \cite{W2015}. Without finite energy condition, $n=1$ was addressed in \cite{CLMR2023}, with partial results for $n\geq 2$ in \cite{CLMR2023} and \cite{FV2023}, as well as in \cite{CMRW2024}. For subcritical case, see \cite{MO2023} \cite{MOW2023}. 

Ivanov, Minchev and Vassilev [\cite{IMV2010} $(n=1)$, \cite{IMV2023} $(n\geq 2)$] classified finite energy positive solutions to the Yamabe type equation on the quaternionic Heisenberg group  using differential identities. They reduced the problem to finding a positive-definite  matrix of size $7\times 7$. For subcritical case, see \cite{LZ}.

For general Heisenberg type group, Garofalo and Vassilev \cite{GV2001} made a conjecture regarding the classification of finite energy positive solutions to the  Yamabe type equation. In the same paper, the conjecture  was verified for all groups of Iwasawa type under the assumption of partial symmetry of the solution. In this paper, we confirm the conjecture of Garofalo and Vassilev on the $15$ dimensional octonionic Heisenberg group.

\begin{theorem}If $u(x,t)$ is a solution of (\ref{1.1}), then up to dilation and translation,     \begin{align}\label{1.2}
     u(x,t)=[(1+|x|^2)^2+|t|^2]^{-5}.
     \end{align}\end{theorem}

 \begin{corollary}
    Every  extremal of the following Sobolev inequality on $15$ dimensional octonionic Heisenberg group must be  (\ref{1.2}), up to dilation and translation. 
 \end{corollary}
 
[Folland and Stein \cite{FS1974}] For $v\in \mathscr{D}^{1,2}(G)$, there is a universal constant $S$ such that

    \begin{align}\label{1.3}
         \bigg(\int_{G}v^\frac{2Q}{Q-2}dV\bigg)^\frac{Q-2}{Q}\leq S\int_{G}|\nabla_G v|^2dV
     \end{align}

For sharp inequalities on Heisenberg type group, one can get the sharp Folland-Stein-Sobolev inequality from Jerison and Lee's result \cite{JL1988} on Heisenberg group and from Ivanov, Minchev and Vassilev's results [\cite{IMV2010}, \cite{IMV2023}] on quaternionic Heisenberg group. Frank and Lieb \cite{FL2012} established sharp Hardy-Littlewood-Sobolev inequalities on the Heisenberg group; see also Hang and Wang \cite{HW2022}. For Hardy-Littlewood-Sobolev inequality on the other Heisenberg type groups, see [\cite{CLZ},\cite{CLZ2}]. Later, Yang \cite{Y2024} gave the optimal constant in Folland-Stein-Sobolev inequality on the Heisenberg type group inspired by the work of Frank and Lieb \cite{FL2012} and Hang and Wang \cite{HW2022}, but he did not classify the minimizers.

 The proof of our result relies on a Jerison-Lee type identity along with a  prior estimates for the solution and its derivatives. The key insight of our approach is the definition (\ref{3.14}) of $B_\nu$, which allows us to analyze Heisenberg type group with several dimensions in the second layer. The main challenge arises from the complexity of the octonion, which is not associative. 
 
\textbf{Convention}: \textit{i,j,.. stand for 0,1,..,7; $\nu,\mu,...$ stand for 1,2,...,7.}

 The paper is organized as follows: In Section \ref{s2}, we collect some notations on the octonionic Heisenberg group. In Section \ref{s3}, we define $\mathcal{D}$ and $B_\nu$, and then we compute a Jerison-Lee type identity. In Section \ref{s4}, we use the Jerison-Lee type identity and a prior estimates to show $\mathcal{D}=B_\nu= 0$ and then we follow the method in \cite{IMV2014} to complete our proof.

\section{Notations}\label{s2}
In this section, we give an introduction to the octonionic Heisenberg group. We consider
\begin{align*}
    G=\mathbb{R}^{8}\times \mathbb{R}^7
\end{align*}
with $\xi=(x,t)=(x^i, t^\nu)$ where $ i=0,1,..,7; \nu=1,2,...,7$ and with the group law: given $\xi=(x,t),\zeta=(y,s)$,
\begin{align*}
    (x,t)\circ(y,s)=(x+y,T),
\end{align*}
where 
\small{
\begin{align*}
    T^1&=t^1+s^1+2(-x^0 y^1+x^1 y^0-x^2 y^3+x^3 y^2-x^4y^5+x^5 y^4+x^6 y^7-x^7 y^6),
\end{align*}
\begin{align*}
    T^2&=t^2+s^2+2(-x^0 y^2+x^1 y^3+x^2 y^0-x^3y^1-x^4y^6-x^5 y^7+x^6 y^4+x^7 y^5),
\end{align*}
\begin{align*}
    T^3&=t^3+s^3+2(-x^0 y^3-x^1 y^2+x^2 y^1+x^3 y^0-x^4y^7+x^5 y^6-x^6 y^5+x^7 y^4),
\end{align*}
\begin{align*}
    T^4&=t^4+s^4+2(-x^0 y^4+x^1 y^5+x^2 y^6+x^3 y^7+x^4 y^0-x^5 y^1-x^6 y^2-x^7 y^3),
\end{align*}
\begin{align*}
    T^5&=t^5+s^5+2(-x^0 y^5-x^1 y^4+x^2 y^7-x^3y^6+x^4y^1+x^5 y^0+x^6 y^3-x^7 y^2),
\end{align*}
\begin{align*}
    T^6&=t^6+s^6+2(-x^0 y^6-x^1 y^7-x^2 y^4+x^3 y^5+x^4y^2-x^5 y^3+x^6 y^0+x^7 y^1),
\end{align*}
\begin{align*}
    T^7&=t^7+s^7+2(-x^0 y^7+x^1 y^6-x^2 y^5-x^3y^4+x^4y^3+x^5 y^2-x^6 y^1+x^7 y^0).
\end{align*}}

We define the norm for $\xi=(x,t)$ as follow:
\begin{align*}
    |\xi|=(|x|^4+|t|^2)^\frac{1}{4},
\end{align*}
with the associated distance function 
\begin{align*}
    d(\xi,\zeta)=|\zeta^{-1}\circ\xi|,\,\,\textrm{for}\,\,\xi,\zeta \in G,
\end{align*}
where $\zeta^{-1}$ is the inverse of $\zeta$ with respect to $\circ$, i.e. $\zeta^{-1}=-\zeta$.
We define
\begin{align*}
    B_R(\xi)=\{\zeta\in G:d(\xi, \zeta)<R\},
\end{align*}
and we denote $B_R=B_R(0)$. We have
\begin{align*}
    |B_R(\xi)|=CR^Q,
\end{align*}
where $C>0$ is a positive constant, $Q=22$ is the homogeneous dimension of $G$ and $|\cdot|$ is the Lebesgue measure.

We define the following $8$ left-invariant vector fields in $G$:
\small{\begin{align*}
&\mathscr{X}^0=\partial_{x^0}+2x^1 \partial_{t^1}+2x^2 \partial_{t^2}+2x^3 \partial_{t^3}+2x^4 \partial_{t^4}+2x^5 \partial_{t^5}+2x^6 \partial_{t^6}+2x^7 \partial_{t^7},\\
&\mathscr{X}^1 =\partial_{x^1}-2x^0 \partial_{t^1}-2x^3 \partial_{t^2}+2x^2 \partial_{t^3}-2x^5 \partial_{t^4}+2x^4 \partial_{t^5}+2x^7 \partial_{t^6}-2x^6  \partial_{t^7},\\
&\mathscr{X}^2=\partial_{x^2}+2x^3 \partial_{t^1}-2x^0 \partial_{t^2}-2x^1 \partial_{t^3}-2x^6\partial_{t^4}-2x^7\partial_{t^5}+2x^4\partial_{t^6}+2x^5 \partial_{t^7},\\
&\mathscr{X}^3=\partial_{x^3}-2x^2\partial_{t^1}+2x^1 \partial_{t^2}-2x^0 \partial_{t^3}-2x^7 \partial_{t^4}+2x^6 \partial_{t^5}-2x^5 \partial_{t^6}+2x^4 \partial_{t^7},\\
&\mathscr{X}^4=\partial_{x^4}+2x^5 \partial_{t^1}+2x^6 \partial_{t^2}+2x^7 \partial_{t^3}-2x^0\partial_{t^4}-2x^1 \partial_{t^5}-2x^2 \partial_{t^6}-2x^3\partial_{t^7},\\
&\mathscr{X}^5=\partial_{x^5}-2x^4 \partial_{t^1}+2x^7 \partial_{t^2}-2x^6 \partial_{t^3}+2x^1 \partial_{t^4}-2x^0\partial_{t^5}+2x^3\partial_{t^6}-2x^2 \partial_{t^7},\\
&\mathscr{X}^6=\partial_{x^6}-2x^7 \partial_{t^1}-2x^4\partial_{t^2}+2x^5\partial_{t^3}+2x^2 \partial_{t^4}-2x^3 \partial_{t^5}-2x^0\partial_{t^6}+2x^1 \partial_{t^7},\\
&\mathscr{X}^7=\partial_{x^7}+2x^6 \partial_{t^1}-2x^5 \partial_{t^2}-2x^4 \partial_{t^3}+2x^3 \partial_{t^4}+2x^2 \partial_{t^5}-2x^1 \partial_{t^6}-2x^0 \partial_{t^7}.   
 \end{align*}}
In fact, the above vector fields are the kernel of the following $7$ contact forms,
\small{\begin{align*}
    \eta_1=&dt^1+2(-x^1 dx^0+x^0 dx^1-x^3 dx^2+x^2 dx^3-x^5dx^4+x^4 dx^5+x^7 dx^6-x^6 dx^7),\\
    \eta_2=&dt^2+2(-x^2 dx^0+x^3 dx^1+x^0 dx^2-x^1dx^3-x^6 dx^4-x^7 dx^5+x^4 dx^6+x^5 dx^7),\\
    \eta_3=&dt^3+2(-x^3dx^0-x^2 dx^1+x^1dx^2+x^0 dx^3-x^7 dx^4+x^6dx^5-x^5 dx^6+x^4 dx^7),\\
    \eta_4=&dt^4+2(-x^4 dx^0+x^5dx^1+x^6 dx^2+x^7 dx^3+x^0 dx^4-x^1dx^5-x^2dx^6-x^3 dx^7),\\
    \eta_5=&dt^5+2(-x^5 dx^0-x^4dx^1+x^7 dx^2-x^6 dx^3+x^1 dx^4+x^0 dx^5+x^3 dx^6-x^2 dx^7),\\
    \eta_6=&dt^6+2(-x^6 dx^0-x^7 dx^1-x^4dx^2+x^5 dx^3+x^2dx^4-x^3 dx^5+x^0 dx^6+x^1 dx^7),\\
    \eta_7=&dt^7+2(-x^7 dx^0+x^6 dx^1-x^5 dx^2-x^4 dx^3+x^3 dx^4+x^2 dx^5-x^1 dx^6+x^0 dx^7).
    \end{align*}}
The following forms $\omega_\nu$ are associated to $d\eta_\nu$,
\small{\begin{align*}
    \omega_1=& dx^0\wedge dx^1+dx^2\wedge dx^3+dx^4\wedge dx^5+dx^7\wedge dx^6,\\
    \omega_2=&dx^0\wedge dx^2+dx^3\wedge dx^1+dx^4\wedge dx^6+dx^5\wedge dx^7,\\
    \omega_3=&dx^0\wedge dx^3+dx^1\wedge dx^2+dx^4\wedge dx^7+dx^6\wedge dx^5,\\
    \omega_4=&dx^0\wedge dx^4+dx^5\wedge dx^1+dx^6\wedge dx^2+dx^7\wedge dx^3,\\
    \omega_5=&dx^0\wedge dx^5+dx^1\wedge dx^4+dx^7\wedge dx^2+dx^3\wedge dx^6,
    \\\omega_6=&dx^0\wedge dx^6+dx^1\wedge dx^7+dx^2\wedge dx^4+dx^5\wedge dx^3,
    \\  
    \omega_7=&dx^0\wedge dx^7+dx^6\wedge dx^1+dx^2\wedge dx^5+dx^3\wedge dx^4.   \end{align*}}

We define some operators related to the octonionic multiplication for convenience:

\begin{equation*}
   \left(
    \begin{array}{ccccccccc}
&\mathscr{X}^0&\mathscr{X}^1&\mathscr{X}^2& \mathscr{X}^3& \mathscr{X}^4&\mathscr{X}^5& \mathscr{X}^6& \mathscr{X}^7\\ 
I_1&\mathscr{X}^1&-\mathscr{X}^0&\mathscr{X}^3& -\mathscr{X}^2& \mathscr{X}^5&-\mathscr{X}^4& -\mathscr{X}^7& \mathscr{X}^6\\ 
I_2&\mathscr{X}^2&-\mathscr{X}^3&-\mathscr{X}^0& \mathscr{X}^1& \mathscr{X}^6&\mathscr{X}^7& -\mathscr{X}^4& -\mathscr{X}^5\\ 
I_3&\mathscr{X}^3&\mathscr{X}^2&-\mathscr{X}^1& -\mathscr{X}^0& \mathscr{X}^7&-\mathscr{X}^6& \mathscr{X}^5& -\mathscr{X}^4\\ 
I_4&\mathscr{X}^4&-\mathscr{X}^5&-\mathscr{X}^6& -\mathscr{X}^7& -\mathscr{X}^0&\mathscr{X}^1& \mathscr{X}^2& \mathscr{X}^3\\
I_5&\mathscr{X}^5&\mathscr{X}^4&-\mathscr{X}^7& \mathscr{X}^6& -\mathscr{X}^1&-\mathscr{X}^0& -\mathscr{X}^3& \mathscr{X}^2\\ 
I_6&\mathscr{X}^6&\mathscr{X}^7&\mathscr{X}^4& -\mathscr{X}^5& -\mathscr{X}^2&\mathscr{X}^3& -\mathscr{X}^0& -\mathscr{X}^1\\ 
I_7&\mathscr{X}^7&-\mathscr{X}^6&\mathscr{X}^5& \mathscr{X}^4& -\mathscr{X}^3&-\mathscr{X}^2& \mathscr{X}^1& -\mathscr{X}^0\\ 
\end{array}
    \right)
    \end{equation*}
For example, $I_5\mathscr{X}^2=I_5I_2\mathscr{X}^0=-I_7\mathscr{X}^0=-\mathscr{X}^7$. 

We list the Lie brackets of the horizontal vector fields.
\begin{lemma}\label{lemma2.1}
We have
  \begin{equation*}
      \begin{aligned}
\big[\mathscr{X}^0,\,\mathscr{X}^1\big]=\big[\mathscr{X}^2,\,\mathscr{X}^3\big]=\big[\mathscr{X}^4,\,\mathscr{X}^5\big]=\big[\mathscr{X}^7,\,\mathscr{X}^6\big]=-4\partial_{t^1},
      \end{aligned}
  \end{equation*}
  \begin{equation*}
      \begin{aligned}
\big[\mathscr{X}^0,\,\mathscr{X}^2\big]=\big[\mathscr{X}^3,\,\mathscr{X}^1\big]=\big[\mathscr{X}^4,\,\mathscr{X}^6\big]=\big[\mathscr{X}^5,\,\mathscr{X}^7\big]=-4\partial_{t^2},
      \end{aligned}
  \end{equation*}
  \begin{equation*}
      \begin{aligned}
\big[\mathscr{X}^0,\,\mathscr{X}^3\big]=\big[\mathscr{X}^1,\,\mathscr{X}^2\big]=\big[\mathscr{X}^4,\,\mathscr{X}^7\big]=\big[\mathscr{X}^6,\,\mathscr{X}^5\big]=-4\partial_{t^3},
      \end{aligned}
  \end{equation*}
  \begin{equation*}
      \begin{aligned}
\big[\mathscr{X}^0,\,\mathscr{X}^4\big]=\big[\mathscr{X}^5,\,\mathscr{X}^1\big]=\big[\mathscr{X}^6,\,\mathscr{X}^2\big]=\big[\mathscr{X}^7,\,\mathscr{X}^3\big]=-4\partial_{t^4},
      \end{aligned}
  \end{equation*}
  \begin{equation*}
      \begin{aligned}
\big[\mathscr{X}^0,\,\mathscr{X}^5\big]=\big[\mathscr{X}^1,\,\mathscr{X}^4\big]=\big[\mathscr{X}^7,\,\mathscr{X}^2\big]=\big[\mathscr{X}^3,\,\mathscr{X}^6\big]=-4\partial_{t^5},
      \end{aligned}
  \end{equation*}
  \begin{equation*}
      \begin{aligned}
\big[\mathscr{X}^0,\,\mathscr{X}^6\big]=\big[\mathscr{X}^1,\,\mathscr{X}^7\big]=\big[\mathscr{X}^2,\,\mathscr{X}^4\big]=\big[\mathscr{X}^5,\,\mathscr{X}^3\big]=-4\partial_{t^6},
      \end{aligned}
  \end{equation*}
  \begin{equation*}
\begin{aligned}
\big[\mathscr{X}^0,\,\mathscr{X}^7\big]=\big[\mathscr{X}^6,\,\mathscr{X}^1\big]=\big[\mathscr{X}^2,\,\mathscr{X}^5\big]=\big[\mathscr{X}^3,\,\mathscr{X}^4\big]=-4\partial_{t^7}.
      \end{aligned}
  \end{equation*}
\end{lemma}

We should be careful since octonionic multiplication is not associative. In general, we don't have $I_\nu I_\mu\mathscr{X}^i=(I_\nu I_\mu)\mathscr{X}^i$. However, we still have the following relations.
\begin{lemma}\label{lemma2.2}
For $i=0,1,...,7$ and $\nu,\mu=1,2,..,7$, we have
    \begin{equation*}
        \begin{aligned}
            I_\nu I_\mu I_\nu \mathscr{X}^i=I_\mu \mathscr{X}^i,\,\,\textrm{for}\,\nu\neq \mu;I_\nu I_\mu I_\nu \mathscr{X}^i=-I_\mu \mathscr{X}^i,\,\,\textrm{for}\,\nu= \mu.        \end{aligned}
    \end{equation*}
\end{lemma}
\begin{lemma}\label{lemma2.3}
    For $\nu,\mu=1,2,...,7$, $\nu\neq \mu$ we have
    \begin{align*}
        \Sigma_{i=0}^7\mathscr{X}^i h I_\nu\mathscr{X}^i h=0,\Sigma_{i=0}^7 I_\mu \mathscr{X}^ih I_\nu\mathscr{X}^i h=0,\Sigma_{i=0}^7 (I_\mu \mathscr{X}^i I_\nu\mathscr{X}^i h+I_\nu\mathscr{X}^iI_\mu\mathscr{X}^ih)=0.    \end{align*}
\end{lemma}

The sub-Laplacian operator on the octonionic Heisenberg group defined as
\begin{align*}
    \Delta_G h:=\Sigma_{j=0}^7(\mathscr{X}^j)^2h.
\end{align*}
And 
\begin{align*}
    |\nabla_G h|^2:=\Sigma_{i=0}^7(\mathscr{X}^j h)^2=\Sigma_{i=0}^7(I_\nu\mathscr{X}^i h)^2,
    \end{align*}
for $\nu=1,2,...,7.$

\section{Jerison-Lee type identity }\label{s3} 
In this section, we derive a Jerison-Lee type identity.

Suppose $u>0$ is a solution of (\ref{1.1}), we set  $u=h^{-\frac{Q-2}{4}}$, then $h$ satisfies
\begin{align}\label{3.1}
    \Delta_G h=6h^{-1}|\nabla_G h|^2+32.
\end{align}
The following function is very important in our proof,
\begin{align}\label{3.2}
    f=4+\frac{|\nabla_G h|^2}{4h}.
    \end{align}

We define $\mathcal{E}(\mathscr{X}^i,\mathscr{X}^j)=0$  and $\mathcal{D}(\mathscr{X}^i,\mathscr{X}^j)$, for $i,j =0,1,...,7.$
\begin{align*}
   \mathcal{E}(\mathscr{X}^i,\mathscr{X}^j)=&\frac{h^{-1}}{8}\bigg[\mathscr{X}^j\mathscr{X}^ih+\Sigma_{\nu=1}^7(I_\nu \mathscr{X}^j)(I_\nu \mathscr{X}^i) h-\Delta_G h \delta_{ij}+8\Sigma_{\nu=1}^7h_{t^\nu}\omega_\nu(\mathscr{X}^i,\mathscr{X}^j)\bigg]\\&-\frac{4h^{-2}}{8}\bigg[\mathscr{X}^jh\mathscr{X}^ih+\Sigma_{\nu=1}^7(I_\nu \mathscr{X}^j)h(I_\nu \mathscr{X}^i) h-|\nabla_G h|^2\delta_{ij}\bigg]\\
   =&\frac{h^{-1}}{8}\bigg[\mathscr{X}^j\mathscr{X}^ih+\Sigma_{\nu=1}^7(I_\nu \mathscr{X}^j)(I_\nu \mathscr{X}^i) h\bigg]\\&-\frac{4h^{-2}}{8}\bigg[\mathscr{X}^jh\mathscr{X}^ih+\Sigma_{\nu=1}^7(I_\nu \mathscr{X}^j)h(I_\nu \mathscr{X}^i) h\bigg]\\&
   -\frac{h^{-1}}{8}(32+2h^{-1}|\nabla_G h|^2)\delta_{ij}+h^{-1}\Sigma_{\nu=1}^7h_{t^\nu}\omega_\nu(\mathscr{X}^i,\mathscr{X}^j)\\=&0,\\
   \mathcal{D}(\mathscr{X}^i,\mathscr{X}^j)=&\frac{h^{-1}}{8}\bigg[7\mathscr{X}^j\mathscr{X}^ih-\Sigma_{\nu=1}^7(I_\nu \mathscr{X}^j)(I_\nu \mathscr{X}^i) h-24\Sigma_{\nu=1}^7h_{t^\nu}\omega_\nu(\mathscr{X}^i,\mathscr{X}^j)\bigg].
    \end{align*}
Then from Lemma \ref{lemma2.1} we have
\begin{equation}\label{3.3}
\begin{aligned}
     &\mathcal{E}(\mathscr{X}^i,\mathscr{X}^j)=\mathcal{E}(\mathscr{X}^j,\mathscr{X}^i)=0,\mathcal{D}(\mathscr{X}^i,\mathscr{X}^j)=\mathcal{D}(\mathscr{X}^j,\mathscr{X}^i),\\
     &\Sigma_{i,j=0}^7\mathcal{D}(\mathscr{X}^i,\mathscr{X}^j)\delta_{ij}=0.    \end{aligned}
\end{equation}
And for $\nu=1,2,...,7$,
\begin{align}\label{3.4}
    \mathcal{D}(\mathscr{X}^i,\mathscr{X}^j)+\Sigma_{\nu=1}^7\mathcal{D}(I_\nu \mathscr{X}^i,I_\nu \mathscr{X}^j)=0.
\end{align}
We denote 
    \begin{align}
        \mathcal{D}^2=\Sigma_{i,j=0}^7\mathcal{D}(\mathscr{X}^i,\mathscr{X}^j)^2.   \end{align}
We also have for $\nu,\mu=1,2,...,7$ from Lemma \ref{lemma2.1},
 \begin{equation}\label{3.6}
 \begin{aligned}
& \Sigma_{i,j=0}^7\mathcal{D}(\mathscr{X}^i,I_\nu \mathscr{X}^j)\omega_\mu(\mathscr{X}^i,\mathscr{X}^j)=-\Sigma_{i=0}^7\mathcal{D}(I_\mu \mathscr{X}^i, I_\nu \mathscr{X}^i)=0.\\&\Sigma_{i,j=0}^7\mathcal{D}(\mathscr{X}^i, I_\nu\mathscr{X}^j)\delta_{ij}=0.
 \end{aligned}
 \end{equation}
 We define  
\begin{align*}
   E(\mathscr{X}^j)=&h^{-1}\Sigma_{i=0}^7\mathcal{E}( \mathscr{X}^j, \mathscr{X}^i)\mathscr{X}^i h   \\=&\frac{h^{-2}}{8}\bigg[\Sigma_{i=0}^7\mathscr{X}^i h\mathscr{X}^i\mathscr{X}^jh+\Sigma_{\nu=1}^7\Sigma_{i=0}^7\mathscr{X}^ihI_\nu \mathscr{X}^i(I_\nu \mathscr{X}^j)h\\&+(-32-6h^{-1}|\nabla_G h|^2)\mathscr{X}^j h\bigg]+h^{-2}\Sigma_{\nu=1}^7h_{t^\nu}(I_\nu \mathscr{X}^j)h\\=&0,\\
    D(\mathscr{X}^j)=&h^{-1}\Sigma_{i=0}^7\mathcal{D}( \mathscr{X}^j, \mathscr{X}^i)\mathscr{X}^i h\\=&\frac{h^{-2}}{8}\bigg[7\Sigma_{i=0}^7\mathscr{X}^i h\mathscr{X}^i\mathscr{X}^j h-\Sigma_{\nu=1}^7 \mathscr{X}^ihI_\nu \mathscr{X}^i(I_\nu \mathscr{X}^j) h\bigg]-3h^{-2}\Sigma_{\nu=1}^7h_{t^\nu}(I_\nu \mathscr{X}^j)h,\\
    F_\nu(\mathscr{X}^j)=&h^{-1}\Sigma_{i=0}^7\mathcal{D}(I_\nu \mathscr{X}^j,I_\nu \mathscr{X}^i)\mathscr{X}^i h,\,\textrm{for}\,\nu=1,2,...,7.
\end{align*}   
Then from (\ref{3.4}) we have
\begin{align}\label{3.7}
    D(\mathscr{X}^j)+\Sigma_{\nu=1}^7 F_\nu (\mathscr{X}^j)=0.
\end{align}
We define

\begin{equation}\label{3.8}
\begin{aligned}
    \mathbb{D}(\mathscr{X}^i,\mathscr{X}^j,\mathscr{X}^k)=h^{-1}\bigg[&D(\mathscr{X}^i,\mathscr{X}^k)\mathscr{X}^jh+\Sigma_{\nu=1}^7D(I_\nu \mathscr{X}^i,\mathscr{X}^k)I_\nu \mathscr{X}^j h\\&+D(\mathscr{X}^j,\mathscr{X}^k)\mathscr{X}^ih+\Sigma_{\nu=1}^7D(I_\nu \mathscr{X}^j,\mathscr{X}^k)I_\nu \mathscr{X}^i h\bigg].
\end{aligned}
\end{equation}

\begin{lemma} \label{lemma 3.1}
For $\nu=1,2,..,7$, we have
\begin{align}\label{3.9}
        \Sigma_{i,j=0}^7(I_\nu \mathscr{X}^i )\mathscr{X}^j h \mathcal{D}(I_\nu \mathscr{X}^i,I_\nu \mathscr{X}^j)= \Sigma_{i,j=0}^7 \mathscr{X}^i \mathscr{X}^j h \mathcal{D}( \mathscr{X}^i,I_\nu \mathscr{X}^j)=0.
    \end{align}
    \end{lemma}
  \begin{proof}
    We  have
    \begin{equation*}
        \begin{aligned}
    &\Sigma_{i,j=0}^7\mathcal{D}(\mathscr{X}^j,\mathscr{X}^i)\mathcal{D}(\mathscr{X}^i,I_\nu \mathscr{X}^j)=\Sigma_{i,j=0}^7\mathcal{D}(\mathscr{X}^i,\mathscr{X}^j)\mathcal{D}(\mathscr{X}^i,I_\nu \mathscr{X}^j)\\=&\Sigma_{i,j=0}^7\mathcal{D}(\mathscr{X}^i,-I_\nu \mathscr{X}^j)\mathcal{D}(\mathscr{X}^i, \mathscr{X}^j)=-\Sigma_{i,j=0}^7\mathcal{D}(\mathscr{X}^i,I_\nu \mathscr{X}^j)\mathcal{D}(\mathscr{X}^j, \mathscr{X}^i)\\=&0.
    \end{aligned}    \end{equation*}
That is
\begin{equation}\label{3.10}
\begin{aligned}
      \Sigma_{i,j=0}^7\bigg[7\mathscr{X}^i\mathscr{X}^j h-\Sigma_{\mu=1}^7(I_\mu \mathscr{X}^i) I_\mu \mathscr{X}^j h-24\Sigma_{\mu=1}^7h_{t^\mu}\omega_\mu (\mathscr{X}^j,\mathscr{X}^i)\bigg]D(\mathscr{X}^i,I_\nu \mathscr{X}^j)=0. 
    \end{aligned}
\end{equation}
We also have
\begin{align}
\mathscr{X}^i\mathscr{X}^jh+\Sigma_{\mu=1}^7(I_\mu \mathscr{X}^j)(I_\mu \mathscr{X}^i)h-\Delta_G h \delta_{ij}+8\Sigma_{\mu=1}^7h_{t^\mu}\omega_\mu (\mathscr{X}^i,\mathscr{X}^j)=0.
\end{align}
Last, from (\ref{3.6}), we have 
\begin{align}
   \Sigma_{\mu=1}^7 \Sigma_{i,j=0}^7 \mathcal{D}(\mathscr{X}^i,I_\nu \mathscr{X}^j)h_{t^\mu}\omega_\mu (\mathscr{X}^i,\mathscr{X}^j)=0, \Sigma_{i,j=0}^7\mathcal{D}(\mathscr{X}^i, I_\nu\mathscr{X}^j)\delta_{ij}=0.\end{align}
Combining these, we get (\ref{3.9}).
 \end{proof}  

  We compute the derivative of $f$ defined in (\ref{3.2}).
  \begin{lemma}\label{lemma3.2} For $i=0,1,...,7,$ we have
  \begin{equation}\label{3.13}       \begin{aligned}
    2\mathscr{X}^i f&=h(D(\mathscr{X}^i)+E(\mathscr{X}^i))- 2\Sigma_{\nu=1}^7h^{-1}h_{t^\nu}(I_\nu \mathscr{X}^i)h+h^{-1}f\mathscr{X}^ih\\&=hD(\mathscr{X}^i)- 2\Sigma_{\nu=1}^7h^{-1}h_{t^\nu}(I_\nu \mathscr{X}^i)h+h^{-1}f\mathscr{X}^ih.\end{aligned}   \end{equation}
          \end{lemma}
\begin{proof}

 First, we have
    \begin{align*}
    D(\mathscr{X}^i)+E(\mathscr{X}^i)=&h^{-2}\Sigma_{j=0}^7\mathscr{X}^jh\mathscr{X}^j\mathscr{X}^ih-2\Sigma_{\nu=1}^7h^{-2}h_{t^\nu}(I_\nu \mathscr{X}^i)h\\&-\frac{3}{4}h^{-3}|\nabla_G h|^2\mathscr{X}^ih-4h^{-2}\mathscr{X}^ih.
\end{align*}
Second,
\begin{align*}
    \Sigma_{j=0}^7\mathscr{X}^j h \mathscr{X}^j\mathscr{X}^ih=\Sigma_{j=0}^7\mathscr{X}^j h \mathscr{X}^i\mathscr{X}^j h+4\Sigma_{\nu=1}^7h_{t^\nu}(I_\nu \mathscr{X}^i) h.
    \end{align*} 
    Last, we have
\begin{align*}
    2\mathscr{X}^if&=-\frac{1}{2}h^{-2}|\nabla_G h|^2\mathscr{X}^ih+h^{-1}\Sigma_{j=0}^7\mathscr{X}^j h \mathscr{X}^i\mathscr{X}^j h\\&=-2h^{-1}f\mathscr{X}^ih+8h^{-1}\mathscr{X}^ih+h^{-1}\Sigma_{j=0}^7\mathscr{X}^jh\mathscr{X}^j\mathscr{X}^ih-4\Sigma_{\nu=1}^7h^{-1}h_{t^\nu}(I_\nu \mathscr{X}^i)h.
    \end{align*}
    Combing these, we have (\ref{3.13}).
    \end{proof}
For $i=0,1,...,7; \nu=1,2,...,7$, we define $B_\nu(\mathscr{X}^i)$ as follow:
\begin{align}\label{3.14}
    B_\nu(\mathscr{X}^i)=2h^{-1}I_\nu \mathscr{X}^i h_{t^\nu}-\Sigma_{\mu=1}^7h^{-2}h_{t^\mu}(I_\mu \mathscr{X}^i)h+\frac{1}{2}h^{-2}f\mathscr{X}^ih.
\end{align}
 We also define
\begin{align}\label{3.15}
    B(\mathscr{X}^i):=\Sigma_{\nu=1}^7 B_\nu (\mathscr{X}^i)=2h^{-1}\Sigma_{\nu=1}^7I_\nu \mathscr{X}^i h_{t^\nu}-7\Sigma_{\mu=1}^7h^{-2}h_{t^\mu}(I_\mu \mathscr{X}^i)h+\frac{7}{2}h^{-2}f\mathscr{X}^ih.\end{align}

We discuss the relation between $\mathcal{D}$ and $B_\nu$.
\begin{lemma}\label{lemma3.3} For $j=0,1,...,7$, we have
\begin{equation}\label{3.16}
     \begin{aligned}
h^{11}\Sigma_{i=0}^7\mathscr{X}^i \bigg[h^{-11}\mathcal{D}(\mathscr{X}^j,\mathscr{X}^i)\bigg]=&3\bigg[\frac{7E(\mathscr{X}^j)-D(\mathscr{X}^j)}{2}+B(\mathscr{X}^j)\bigg]\\=&-\frac{3}{2}D(\mathscr{X}^j)+3B(\mathscr{X}^j).
        \end{aligned}\end{equation}
   \end{lemma}
\begin{proof}
We start with
    \begin{align*}
    &h^{11}\Sigma_{i=0}^7\mathscr{X}^i\bigg[h^{-11}\mathcal{D}(\mathscr{X}^j,\mathscr{X}^i)\bigg]\\=&\frac{h^{-1}}{8}\bigg(7\Sigma_{i=0}^7   
    (\mathscr{X}^i)^2\mathscr{X}^jh-        
    \Sigma_{i=0}^7    
    \Sigma_{\nu=1}^7\mathscr{X}^i(I_\nu\mathscr{X}^i)(I_\nu \mathscr{X}^j)h-24\Sigma_{\nu=1}^7(I_\nu \mathscr{X}^j)h_{t^\nu}\bigg)-12D(\mathscr{X}^j).
\end{align*}
From Lemma \ref{lemma2.1}, we have
\begin{align*}
    \Sigma_{i=0}^7\mathscr{X}^i (I_\nu \mathscr{X}^i)=-16\partial_{t^\nu},
\end{align*}
and
\begin{align*}
    7\Sigma_{i=0}^7 (\mathscr{X}^i)^2\mathscr{X}^jh=7\Sigma_{i=0}^7\mathscr{X}^j(\mathscr{X}^i)^2 h+56 \Sigma_{\nu=1}^7(I_\nu \mathscr{X}^j)h_{t^\nu} .
    \end{align*}
    Using (\ref{3.1}) and Lemma \ref{lemma3.2}, we have
    \begin{align*}
       & 7\mathscr{X}^j(\Delta_G h)=168\mathscr{X}^jf\\=&84h(D(\mathscr{X}^j)+E(\mathscr{X}^j))-168\Sigma_{\nu=1}^7h^{-1}h_{t^\nu}(I_\nu \mathscr{X}^j)h+84h^{-1}f\mathscr{X}^jh.
    \end{align*}
  Combining these, we have (\ref{3.16}).  
    \end{proof}

 To derive Jerison-Lee type identity, we need the following propositions.        
\begin{proposition}
    \begin{equation}
    \begin{aligned}
O_1=&h^{11}\Sigma_{i,j=0}^7\mathscr{X}^i\bigg[h^{-11}f\mathscr{X}^j h\mathcal{D}(\mathscr{X}^i,\mathscr{X}^j)\bigg]\\=& fh\mathcal{D}^2+\Sigma_{i=0}^7D(\mathscr{X}^i)\bigg[\frac{1}{2}h^2D(\mathscr{X}^i)-\Sigma_{\nu=1}^7h_{t^\nu}I_\nu \mathscr{X}^i h+\frac{1}{2}f\mathscr{X}^i h\bigg]\\&+\Sigma_{i=0}^7 f\mathscr{X}^i h\bigg[-\frac{3}{2}D(\mathscr{X}^i)+3B(\mathscr{X}^i)\bigg] 
\end{aligned}
    \end{equation}
\end{proposition}
\begin{proof}
We use Lemma \ref{lemma3.2} and Lemma \ref{lemma3.3} with
    \begin{equation*}
    \begin{aligned}
&h^{11}\Sigma_{i,j=0}^7\mathscr{X}^i\bigg[h^{-11}f\mathscr{X}^j h\mathcal{D}(\mathscr{X}^i,\mathscr{X}^j)\bigg]\\=&h^{11}\Sigma_{i,j=0}^7     f\mathscr{X}^j h\mathscr{X}^i\bigg[h^{-11}\mathcal{D}(\mathscr{X}^i,\mathscr{X}^j)\bigg]+f\Sigma_{i,j=0}^7\mathscr{X}^i\mathscr{X}^j h \mathcal{D}(\mathscr{X}^i,\mathscr{X}^j)+\Sigma_{i=0}^7h\mathscr{X}^if D(\mathscr{X}^i).
\end{aligned}
    \end{equation*}
    It suffices to prove
\begin{align*}
    &h^{-1}\Sigma_{i,j=0}^7\mathscr{X}^i\mathscr{X}^j h \mathcal{D}(\mathscr{X}^i,\mathscr{X}^j)\\=&\frac{h^{-1}}{8}\Sigma_{i,j=0}^7\bigg[7\mathscr{X}^i\mathscr{X}^j h-\Sigma_{\nu=1}^7(I_\nu \mathscr{X}^i )(I_\nu \mathscr{X}^j)h-24\Sigma_{\nu=1}^7h_{t^\nu}\omega_\nu (\mathscr{X}^j,\mathscr{X}^i) \bigg]\mathcal{D}(\mathscr{X}^i,\mathscr{X}^j)    \\=&\mathcal{D}^2.
    \end{align*}    
  It follows from
  \begin{align*}
       &\frac{1}{8}h^{-1}\Sigma_{i,j=0}^7\bigg[\mathscr{X}^i\mathscr{X}^j h+
       \Sigma_{\nu=1}^7(I_\nu \mathscr{X}^i )(I_\nu \mathscr{X}^j)h       
       \bigg]\mathcal{D}(\mathscr{X}^i,\mathscr{X}^j)\\=&\frac{1}{8}h^{-1}\Sigma_{i,j=0}^7 \mathscr{X}^i\mathscr{X}^jh\bigg[\mathcal{D}(\mathscr{X}^i,\mathscr{X}^j)+\Sigma_{\nu=1}^7\mathcal{D}(I_\nu \mathscr{X}^i,I_\nu \mathscr{X}^j)    \bigg]\\=&0       
       \end{align*}
       and
       \begin{align*}
\Sigma_{\nu=1}^7\Sigma_{i,j=0}^7\mathcal{D}(\mathscr{X}^i,\mathscr{X}^j)\omega_\nu(\mathscr{X}^j,\mathscr{X}^i)h_{t^\nu}=0.
\end{align*}

\end{proof}

\begin{proposition}
    \begin{equation}\label{3.18}
    \begin{aligned}
O_2=&h^{11}\Sigma_{\nu=1}^7\Sigma_{i,j=0}^7I_\nu \mathscr{X}^i\bigg[h^{-11}h_{t^\nu}\mathscr{X}^j h\mathcal{D}(I_\nu \mathscr{X}^i,I_\nu \mathscr{X}^j)\bigg]\\=&\Sigma_{i=0}^7\bigg[\Sigma_{\nu=1}^7\frac{1}{2}h^2F_\nu(\mathscr{X}^i)B_\nu (\mathscr{X}^i)-\frac{1}{2}\Sigma_{\mu=1}^7h_{t^\mu} D(\mathscr{X}^i) (I_\mu\mathscr{X}^i)h +\frac{1}{4}fD(\mathscr{X}^i)\mathscr{X}^i h \bigg]\\&
-\Sigma_{\nu=1}^7\Sigma_{i=0}^7h_{t^\nu}I_\nu\mathscr{X}^i h\bigg[\frac{-3D( \mathscr{X}^i)}{2}+3B(\mathscr{X}^i)\bigg]\end{aligned}
    \end{equation}
\end{proposition}
\begin{proof}
Using Lemma \ref{3.1}, Lemma \ref{3.3} and (\ref{3.7}), we have
     \begin{align*}
&h^{11}\Sigma_{\nu=1}^7\Sigma_{i,j=0}^7I_\nu \mathscr{X}^i\bigg[h^{-11}h_{t^\nu}\mathscr{X}^j h\mathcal{D}(I_\nu \mathscr{X}^i,I_\nu \mathscr{X}^j)\bigg]\\=&\Sigma_{\nu=1}^7\Sigma_{i=0}^7hF_\nu(\mathscr{X}^i) \bigg[\frac{1}{2} h B_\nu (\mathscr{X}^i)+\frac{1}{2}\Sigma_{\mu=1}^7 h^{-1}h_{t^\mu} (I_\mu\mathscr{X}^i)h -\frac{1}{4}h^{-1}f\mathscr{X}^i h \bigg]\\&+\Sigma_{\nu=1}^7\Sigma_{j=0}^7h_{t^\nu}\mathscr{X}^j h\bigg[\frac{-3D(I_\nu \mathscr{X}^j)}{2}+3B(I_\nu\mathscr{X}^j)\bigg] \\=&\Sigma_{i=0}^7\bigg[\Sigma_{\nu=1}^7\frac{1}{2}h^2F_\nu(\mathscr{X}^i)B_\nu (\mathscr{X}^i)-\frac{1}{2}\Sigma_{\mu=1}^7 h_{t^\mu}D(\mathscr{X}^i) (I_\mu\mathscr{X}^i)h +\frac{1}{4}fD(\mathscr{X}^i)\mathscr{X}^i h \bigg]\\&
-\Sigma_{\nu=1}^7\Sigma_{i=0}^7h_{t^\nu}I_\nu\mathscr{X}^i h\bigg[\frac{-3D( \mathscr{X}^i)}{2}+3B(\mathscr{X}^i)\bigg]. 
\end{align*}

\end{proof}

\begin{proposition}
\begin{equation}\label{3.19}
    \begin{aligned}
       O_3=& h^{11}\Sigma_{\nu=1}^7\Sigma_{i=0}^7I_\nu\mathscr{X}^i \bigg[h^{-10}B_\nu(\mathscr{X}^i)h_{t^\nu}\bigg]   \\=&\frac{1}{2}h^2\Sigma_{\nu=1}^7B_\nu^2 +\Sigma_{i=0}^7\bigg[-\frac{1}{4}\Sigma
_{\nu=1}^7D(\mathscr{X}^i)I_\nu\mathscr{X}^i h h_{t^\nu}+\Sigma_{\nu=1}^7B(\mathscr{X}^i)I_\nu{\mathscr{X}^i}hh_{t^\nu}\\&-\frac{1}{4}fB(\mathscr{X}^i)\mathscr{X}^i h\bigg]
       \end{aligned} 
       \end{equation}\end{proposition}

\begin{proof}
From the definition of (\ref{3.14}), we have
\begin{equation}
    \begin{aligned}
    &\Sigma_{\nu=1}^7\Sigma_{i=0}^7h^{11}I_\nu\mathscr{X}^i \bigg[h^{-10}B_\nu(\mathscr{X}^i)h_{t^\nu}\bigg]\\=&-12\Sigma_{\nu=1}^7\Sigma_{i=0}^7h_{t^\nu}B_\nu(\mathscr{X}^i)I_\nu\mathscr{X}^i h+\Sigma_{\nu=1}^7\Sigma_{i=0}^7 hI_\nu\mathscr{X}^i h_{t^\nu}B_\nu(\mathscr{X}^i)\\&+h^{-1}\Sigma_{\nu=1}^7h_{t^\nu}\Sigma_{i=0}^7\bigg[2I_\nu\mathscr{X}^i h I_\nu\mathscr{X}^i h_{t^\nu}+2h I_\nu\mathscr{X}^i I_\nu\mathscr{X}^i h_{t^\nu}-\Sigma_{\mu=1}^7I_\nu\mathscr{X}^ih_{t^\mu}I_\mu\mathscr{X}^i h\\&-\Sigma_{\mu=1}^7h_{t^\mu}I_\nu\mathscr{X}^iI_\mu\mathscr{X}^i h+\frac{1}{2}I_\nu\mathscr{X}^i f\mathscr{X}^i h+\frac{1}{2}fI_\nu\mathscr{X}^i\mathscr{X}^i h\bigg].
\end{aligned}\end{equation}
First, we should be careful here since the octonionic multiplication is not associative, we have
\begin{equation}
    \begin{aligned}
&h^{-1}\Sigma_{\nu=1}^7h_{t^\nu}\Sigma_{i=0}^7\bigg[2I_\nu\mathscr{X}^i h I_\nu\mathscr{X}^i h_{t^\nu}-\Sigma_{\mu=1}^7I_\nu\mathscr{X}^i h_{t^\mu}I_\mu\mathscr{X}^i h\bigg]\\=&h^{-1}\Sigma_{\nu=1}^7h_{t^{\nu}}\Sigma_{i=0}^7\bigg[I_\nu\mathscr{X}^i h I_\nu\mathscr{X}^i h_{t^\nu}-\Sigma_{\mu\neq \nu,\mu=1}^7I_\nu\mathscr{X}^i h_{t^\mu}I_\mu\mathscr{X}^i h \bigg]\\=&h^{-1}\Sigma_{\nu=1}^7h_{t^\nu}\Sigma_{i=0}^7 \Sigma_{\mu=1}^7I_{\mu}\mathscr{X}^i h_{t^\mu}I_\nu\mathscr{X}^i h\\=&\frac{1}{2}\Sigma_{\nu=1}^7h_{t^\nu}\Sigma_{i=0}^7\Sigma_{\mu=1}^7 B_\mu(\mathscr{X}^i)I_\nu\mathscr{X}^i h+\frac{7}{2}\Sigma_{\nu=1}^7h^{-2}h_{t^\nu}^2|\nabla_G h|^2.
 \end{aligned}
\end{equation}
The second equality holds by Lemma \ref{lemma2.2}, for $\nu\neq \mu$,
\begin{align*}
    I_\nu(-I_\nu I_\mu\mathscr{X}^i)=I_\mu \mathscr{X}^i, I_\mu(-I_\nu I_\mu\mathscr{X}^i)=-I_\nu\mathscr{X}^i,
\end{align*}
then we replace $\mathscr{X}^i$ by $-I_\nu I_\mu \mathscr{X}^i$ to get
\begin{align*}
\Sigma_{i=0}^7I_\nu\mathscr{X}^i h_{t^\mu}I_\mu\mathscr{X}^i h=\Sigma_{i=0}^7 I_\nu(-I_\nu I_\mu\mathscr{X}^i) h_{t^\mu}I_\mu(-I_\nu I_\mu\mathscr{X}^i) h=-\Sigma_{i=0}^7I_\mu \mathscr{X}^i h_{t^\mu}I_\nu \mathscr{X}^i h.
\end{align*}
The last equality is given by
\begin{align*}
    &\Sigma_{i=0}^7I_\mu\mathscr{X}^i h_{t^\mu}I_\nu\mathscr{X}^ih\\=&\Sigma_{i=0}^7\bigg[\frac{1}{2}hB_\mu(\mathscr{X}^i)I_\nu\mathscr{X}^i h+\frac{1}{2}\Sigma_{\tau=1}^7h^{-1}h_{t^\tau}I_\tau \mathscr{X}^i hI_\nu\mathscr{X}^i h-\frac{1}{4}h^{-1}f\mathscr{X}^i h I_\nu\mathscr{X}^i h\bigg]\\=&\Sigma_{i=0}^7\bigg[\frac{1}{2}hB_\mu(\mathscr{X}^i)I_\nu\mathscr{X}^i h+\frac{1}{2}h^{-1}h_{t^\nu}I_\nu \mathscr{X}^i hI_\nu\mathscr{X}^i h\bigg],\end{align*}
since from Lemma \ref{lemma2.3}, $\Sigma_{i=0}^7I_\nu \mathscr{X}^i h I_\mu \mathscr{X}^i h=0$, for $\nu\neq \mu$ and $\Sigma_{i=0}^7I_\nu\mathscr{X}^i h\mathscr{X}^i h=0.$

Next, we have
\begin{align*}
    \frac{1}{2}\Sigma_{i=0}^7fI_\nu\mathscr{X}^i\mathscr{X}^i h=8fh_{t^\nu},
\end{align*}
and
\begin{align*}
    2h\Delta_G h_{t^\nu}=48 h f_{t^\nu}=12h\Sigma_{i=0}^7B_\nu(\mathscr{X}^i)I_\nu\mathscr{X}^i h,\end{align*}
    since
\begin{align}\label{ft}
    f_{t^\nu}=-\frac{1}{4}h^{-2}h_{t^\nu}|\nabla_G h|^2+\frac{1}{2}h^{-1}\Sigma_{i=0}^7  I_\nu\mathscr{X}^i h I_\nu\mathscr{X}^i h_{t^\nu}=\frac{1}{4}\Sigma_{i=0}^7B_\nu(\mathscr{X}^i)I_\nu\mathscr{X}^i h.
    \end{align}
We also have
\begin{align*}
    -h^{-1}\Sigma_{i=0}^7\Sigma_{\nu,\mu=1}^7h_{t^\nu}h_{t^\mu}I_\nu \mathscr{X}^i I_\mu\mathscr{X}^i h=-h^{-1}\Sigma_{\nu=1}^7h^2_{t^\nu}\Delta_G h=h^{-1}\Sigma_{\nu=1}^7\bigg[-24fh_{t^\nu}^2+64h_{t^\nu}^2\bigg],
    \end{align*}
this is because for $\nu\neq \mu$, by Lemma \ref{lemma2.3}, $\Sigma_{i=0}^7(I_\nu \mathscr{X}^iI_\mu \mathscr{X}^i h+I_\mu\mathscr{X}^iI_\nu\mathscr{X}^ih)=0.$

    Using Lemma \ref{lemma3.2}, we have
\begin{align*}
   & \frac{1}{2}h^{-1}\Sigma_{\nu=1}^7h_{t^\nu}\Sigma_{i=0}^7I_\nu\mathscr{X}^i f\mathscr{X}^i h\\=&h^{-1}\Sigma_{\nu=1}^7h_{t^\nu}\bigg[\frac{1}{4}\Sigma_{i=0}^7hD(I_\nu\mathscr{X}^i)\mathscr{X}^i h +\frac{1}{2}h^{-1}h_{t^\nu}|\nabla_G h|^2\bigg]\\=&h^{-1}\Sigma_{\nu=1}^7h_{t^\nu}\bigg[-\frac{1}{4}\Sigma_{i=0}^7hD(\mathscr{X}^i)I_\nu\mathscr{X}^i h +\frac{1}{2}h^{-1}h_{t^\nu}|\nabla_G h|^2\bigg].\end{align*}
   Notice that
\begin{align*}
    64h^{-1}h^2_{t^\nu}+\frac{1}{2}h^{-2}h^2_{t^\nu}|\nabla_G h|^2=16h^{-1}h^2_{t^\nu}f-\frac{7}{2}h^{-2}h^2_{t^\nu}|\nabla_G h|^2,
    \end{align*}
we arrive at
\begin{align*}
   &h^{-1}\Sigma_{\nu=1}^7h_{t^\nu}\Sigma_{i=0}^7\bigg[2I_\nu\mathscr{X}^i h I_\nu\mathscr{X}^i h_{t^\nu}+2h I_\nu\mathscr{X}^i I_\nu\mathscr{X}^i h_{t^\nu}-\Sigma_{\mu=1}^7I_\nu\mathscr{X}^i h_{t^\mu}I_\mu\mathscr{X}^i_\alpha h\\&-\Sigma_{\mu=1}^7h_{t^\mu}I_\nu\mathscr{X}^i I_\mu\mathscr{X}^i h+\frac{1}{2}I_\nu\mathscr{X}^i f\mathscr{X}^i h+\frac{1}{2}fI_\nu\mathscr{X}^i\mathscr{X}^i h\bigg]\\=&\Sigma_{\nu=1}^7h_{t^\nu}\Sigma_{i=0}^7\bigg[-\frac{1}{4}D(\mathscr{X}^i )I_\nu\mathscr{X}^i h+12B_\nu(\mathscr{X}^i)I_\nu\mathscr{X}^ih+\frac{1}{2}\Sigma_{\mu=1}^7B_\mu(\mathscr{X}^i)I_\nu\mathscr{X}^i h\bigg].
    \end{align*}
Last, we have
\begin{align*}
    &\Sigma_{\nu=1}^7\Sigma_{i=0}^7hI_\nu\mathscr{X}^i h_{t^\nu}B_\nu(\mathscr{X}^i)\\=&\Sigma_{\nu=1}^7\Sigma_{i=0}^7hB_\nu(\mathscr{X}^i)\bigg[\frac{1}{2}hB_\nu(\mathscr{X}^i)+\frac{1}{2}\Sigma_{\mu=1}^7h^{-1}h_{t^\mu}(I_\mu\mathscr{X}^i)h-\frac{1}{4}h^{-1}f\mathscr{X}^i h\bigg]\\=&\frac{1}{2}\Sigma_{\nu=1}^7\Sigma_{i=0}^7h^2B_\nu(\mathscr{X}^i)^2+\Sigma_{i=0}^7\bigg[\frac{1}{2}\Sigma_{\mu=1}^7h_{t^\mu}(I_\mu\mathscr{X}^i)h B(\mathscr{X}^i)-\frac{1}{4}fB(\mathscr{X}^i )\mathscr{X}^i h\bigg].
\end{align*}
Combining these, we complete the proof.
\end{proof}
\begin{theorem}[Jerison-Lee type identity]We have

  \begin{equation}\label{jl}
     \begin{aligned}
        &\frac{1}{3}O_1+\frac{4}{3}O_2+4O_3\\=& \frac{1}{3}fh\mathcal{D}^2+\frac{1}{6}\Sigma_{i=0}^7h^2 D(\mathscr{X}^i)^2+
        \Sigma_{i=0}^7\Sigma_{\nu=1}^7h^2\bigg[\frac{2}{3}F_\nu(\mathscr{X}^i)B_\nu (\mathscr{X}^i)+2 B_\nu(\mathscr{X}^i)^2\bigg]\\\geq & \frac{1}{1000}\bigg[fh\mathcal{D}^2+\Sigma_{i=0}^7\Sigma_{\nu=1}^7h^2B_\nu (\mathscr{X}^i)^2\bigg]
     \end{aligned}
 \end{equation}\end{theorem}
\begin{proof}

From the definition (\ref{3.8}), we have
\begin{align}\label{*}
    \Sigma_{i,j,k=0}^7\mathbb{D}(\mathscr{X}^i,\mathscr{X}^j,\mathscr{X}^k)^2=16h^{-2}\mathcal{D}^2|\nabla h|^2+16\Sigma_{i=0}^7[D(\mathscr{X}^i)^2-\Sigma_{\nu=1}^7F_\nu(\mathscr{X}^i)^2]
\end{align}

This is because from Lemma \ref{lemma2.3}, we have
\begin{align*}
    \Sigma_{i,j,k=0}^7 \mathcal{D}(\mathscr{X}^i,\mathscr{X}^k)\mathscr{X}^j h\mathcal{D}(I_\nu\mathscr{X}^i,\mathscr{X}^k)I_\nu\mathscr{X}^j h=0\end{align*}
and 

\begin{align*}
    \Sigma_{i,j,k=0}^7 \mathcal{D}(I_\mu\mathscr{X}^i,\mathscr{X}^k)I_\mu\mathscr{X}^j h\mathcal{D}(I_\nu\mathscr{X}^i,\mathscr{X}^k)I_\nu\mathscr{X}^j h=0.\end{align*}

By the definition of $D(\mathscr{X}^i)$ and $F_\nu(\mathscr{X}^i)$, we have
\begin{align*}
    h^{-2}\Sigma_{i,j,k=0}^7\mathcal{D}(\mathscr{X}^i,\mathscr{X}^k)\mathscr{X}^j h \mathcal{D}(\mathscr{X}^j,\mathscr{X}^k)\mathscr{X}^i h=\Sigma_{i=0}^7D(\mathscr{X}^i)^2,\end{align*}
\begin{align*}
    &h^{-2}\Sigma_{i,j,k=0}^7\mathcal{D}(\mathscr{X}^i,\mathscr{X}^k)\mathscr{X}^j h \mathcal{D}(I_\nu\mathscr{X}^j,\mathscr{X}^k)I_\nu\mathscr{X}^i h\\=&-h^{-2}\Sigma_{i,j,k=0}^7\mathcal{D}(I_\nu\mathscr{X}^i,I_\nu\mathscr{X}^k)\mathscr{X}^j h \mathcal{D}(I_\nu\mathscr{X}^j,I_\nu\mathscr{X}^k)\mathscr{X}^i h\\=&-\Sigma_{i=0}^7F_\nu(\mathscr{X}^i)^2,\end{align*}
and
\begin{align*}
    &h^{-2}\Sigma_{i,j,k=0}^7\mathcal{D}(I_\nu \mathscr{X}^i,\mathscr{X}^k)I_\nu\mathscr{X}^j h \mathcal{D}(I_\nu \mathscr{X}^j,\mathscr{X}^k)I_\nu\mathscr{X}^i h\\=& h^{-2}\Sigma_{i,j,k=0}^7\mathcal{D}(\mathscr{X}^i,\mathscr{X}^k)\mathscr{X}^j h \mathcal{D}(\mathscr{X}^j,\mathscr{X}^k)\mathscr{X}^i h\\=&\Sigma_{i=0}^7D(\mathscr{X}^i)^2.\end{align*}
We should be careful to deal with the rest terms because of the non-associativity.
\begin{align*}
    &h^{-2}\Sigma_{\nu\neq \mu,\nu=1}^7\Sigma_{i,j,k=0}^7\mathcal{D}(I_\mu \mathscr{X}^i,\mathscr{X}^k)I_\mu \mathscr{X}^j h \mathcal{D}(I_\nu\mathscr{X}^j,\mathscr{X}^k)I_\nu\mathscr{X}^i h\\=&h^{-2}\Sigma_{\nu\neq \mu,\nu=1}^7\Sigma_{i,j,k=0}^7\mathcal{D}(I_\mu I_\nu \mathscr{X}^i,\mathscr{X}^k) \mathscr{X}^i h \mathcal{D}(I_\nu I_\mu \mathscr{X}^j,\mathscr{X}^k)\mathscr{X}^j h\\=  
    &-\Sigma_{\nu\neq \mu, \nu=1}^7\Sigma_{i=0}^7F_\nu(\mathscr{X}^i)^2,\end{align*}
since for $\nu=1,\mu=2$,
we have $I_1I_2\mathscr{X}^i=\delta I_3\mathscr{X}^i$ and $I_2I_1\mathscr{X}^i=-\delta I_3\mathscr{X}^i$ for $i=0,1,,...,7$, where $\delta=1$ or $-1$. Then
\begin{align*}
&h^{-2}\Sigma_{i,j,k=0}^7\mathcal{D}(I_2 I_1 \mathscr{X}^i,\mathscr{X}^k) \mathscr{X}^i h \mathcal{D}(I_1 I_2 \mathscr{X}^j,\mathscr{X}^k)\mathscr{X}^j h\\=& -\delta^2h^{-2}\Sigma_{i,j,k=0}^7\mathcal{D}(I_3\mathscr{X}^i,I_3\mathscr{X}^k)\mathscr{X}^i hD(I_3\mathscr{X}^j,I_3\mathscr{X}^k)\mathscr{X}^jh=-\Sigma_{i=0}^7F_3(\mathscr{X}^i)^2
\end{align*}

    To prove (\ref{jl}), we use (\ref{*}), that is
    \begin{align*}
        h^{-2}\mathcal{D}^2|\nabla h|^2+\Sigma_{i=0}^7[D(\mathscr{X}^i)^2-\Sigma_{\nu=1}^7F_\nu(\mathscr{X}^i)^2]\geq 0.    \end{align*}
        Then
\begin{align*}
        \frac{33}{100}fh\mathcal{D}^2+\frac{33}{400}\Sigma_{i=0}^7h^2D(\mathscr{X}^i)^2   \geq  \frac{33}{400}h^2\Sigma_{i=0}^7\Sigma_{\nu=1}^7F_\nu(\mathscr{X}^i)^2
    \end{align*}
 and 
 \begin{align*}
     \frac{33}{400}F_\nu(\mathscr{X}^i)^2+2B_\nu(\mathscr{X}^i)^2+\frac{2}{3}F_\nu(\mathscr{X}^i)B_\nu (\mathscr{X}^i)\geq \frac{1}{1000}B_\nu (\mathscr{X}^i)^2
 \end{align*}
 give (\ref{jl}).
\end{proof}

\section{Proof of the main result }\label{s4}

In this section, we use (\ref{jl}) to prove our main result.
\subsection{Decay} We state the decay of $u$:     
\begin{proposition}\cite{L2019}
    If $u$ satisfies (\ref{1.1}), then there are positive constants $c, C$ such that
    \begin{align}
        c|\xi|^{2-Q}\leq u(\xi)\leq C|\xi|^{2-Q}.
    \end{align}
    \end{proposition}
    For $R>1$, we consider $v(\xi)=R^{Q-2}u(R\xi)$, we have that $v$ is bounded in $B_6\backslash B_1$ and $v(\xi)$ satisfies
    \begin{align*}
        \Delta_G v+cR^{-2}v^\frac{Q+2}{Q-2}=0\,\,\textrm{in}\,B_6\backslash B_1
    \end{align*}
By local regularity theory, see \cite{F1975}, we have $v(\xi)\in \Gamma^{1,\alpha}$ in $B_5\backslash B_2$. A bootstrap argument shows that $v(\xi) \in \Gamma^{\infty}$ in $B_4\backslash B_3$. Therefore for all horizontal vector fields $\mathscr{X}^i, \mathscr{X}^j, \mathscr{X}^k$, when $|\xi|\rightarrow \infty$, we have 
\begin{align}\label{est}
    |\mathscr{X}^iu(\xi)|\leq C|\xi|^{1-Q},\,\,|\mathscr{X}^i\mathscr{X}^j u(\xi)|\leq C |\xi|^{-Q}\,\,|\mathscr{X}^i\mathscr{X}^j\mathscr{X}^k u(\xi)|\leq C|\xi|^{-Q-1}.  
\end{align}
\subsection{Integration estimates}
From (\ref{est}), as $|\xi|\rightarrow\infty$, we have
\begin{align*}
    &f\leq C |\xi|^2,h +|\xi||\mathscr{X}^i h|+|\xi|^2|\mathscr{X}^i\mathscr{X}^j h|+|\xi|^3|\mathscr{X}^i\mathscr{X}^j\mathscr{X}^k h|   \leq C |\xi|^4,\\& |\mathcal{D}(\mathscr{X}^i,\mathscr{X}^j)|\leq C|\xi|^{-2},|B_\nu(\mathscr{X}^i)|\leq C|\xi|^{-3}.
\end{align*}
Therefore, by (\ref{jl}),
\begin{equation*}
\begin{aligned}
     &\frac{1}{1000}\int_{B_R}h^{-11}\bigg[fh\mathcal{D}^2+\Sigma_{i=0}^7\Sigma_{\nu=1}^7h^2B_\nu (\mathscr{X}^i)^2\bigg]     
     \\\leq& \int_{B_R} h^{-11}\bigg(\frac{1}{3} O_1+\frac{4}{3}O_2+4O_3\bigg)=O(R^{2-Q})\rightarrow 0, \,\,\textrm{as}\,R\rightarrow \infty.\end{aligned}
    \end{equation*}
We conclude that for all  $\mathscr{X}^i$ and $ \mathscr{X}^j$, 
\begin{align}\label{4.3}
    B_\nu (\mathscr{X}^i)=\mathcal{D}(\mathscr{X}^i,\mathscr{X}^j)=\mathcal{E}(\mathscr{X}^i,\mathscr{X}^j)=0.
\end{align}

\subsection{The conclusion}
Using (\ref{4.3}), we follow \cite{IMV2014} closely to get our conclusion.
\begin{lemma}\label{lemma 4.2}
For any $i,j=0,1,..7,$ $\mathcal{D}(\mathscr{X}^i,\mathscr{X}^j)=0$ implies 
   \begin{align}
       \mathscr{X}^i\mathscr{X}^j h=\frac{1}{2}[ \mathscr{X}^i,  \mathscr{X}^j]h\,\,\,\textrm{for}\,i\neq j.
   \end{align}
   
\end{lemma}
\begin{proof}
    We only consider $i=1,j=0$ here, the other cases can also be proved in the same way.
    
    $\mathcal{D}(\mathscr{X}^0,\mathscr{X}^1)=0$ gives
    \begin{align*}
        &7\mathscr{X}^1\mathscr{X}^0 h+\mathscr{X}^0 \mathscr{X}^1 h+\mathscr{X}^3 \mathscr{X}^2 h-\mathscr{X}^2 \mathscr{X}^3 h\\&+\mathscr{X}^5\mathscr{X}^4 h-\mathscr{X}^4 \mathscr{X}^5 h-\mathscr{X}^7 \mathscr{X}^6 h+\mathscr{X}^6 \mathscr{X}^7h=24h_{t^1}.        
        \end{align*}
        That is
        \begin{align*}
            8\mathscr{X}^1\mathscr{X}^0 h-[\mathscr{X}^1,\mathscr{X}^0]h+[\mathscr{X}^3,\mathscr{X}^2]h+[\mathscr{X}^5,\mathscr{X}^4]h+[\mathscr{X}^6,\mathscr{X}^7]h=24h_{t^1}.            \end{align*}
            By Lemma \ref{lemma2.1}, we have $\mathscr{X}^1 \mathscr{X}^0 h=2h_{t^1}.$
            
\end{proof}
\begin{lemma}\label{lemma 4.3}
For any $i,j=0,1,..7$,
 $\mathcal{D}(\mathscr{X}^i,\mathscr{X}^j)=0$ implies
    \begin{align}\label{4.5}
        h_{t^\nu t^\nu}=2\kappa,\,\,\, h_{t^\nu t^\mu}=0,\, \nu \neq \mu,
    \end{align}
    for $\nu,\mu=1,2,..,7$, where $\kappa$ is a positive constant. In particular,
    \begin{align}\label{4.6}
        h(x,t)=g(x)+\kappa\Sigma_{\nu=1}^7(t^\nu+s^\nu(x))^2,
    \end{align}
    for some functions $g(x), s^\nu(x),\nu=1,2,...,7$ on $\mathbb{R}^{8}$. 
\end{lemma}
\begin{proof}

From
    \begin{align*}
        &8h_{t^1t^2}=-4\mathscr{X}^0\mathscr{X}^1 h_{t^2}=-4\mathscr{X}^0\partial_{t^2}\mathscr{X}^1 h=\mathscr{X}^0[\mathscr{X}^0,\mathscr{X}^2]\mathscr{X}^1 h\\=&(\mathscr{X}^0)^2\mathscr{X}^2\mathscr{X}^1 h-\mathscr{X}^0\mathscr{X}^2\mathscr{X}^0\mathscr{X}^1 h =2(\mathscr{X}^0)^2h_{t^3}-4h_{t^1t^2}
        \end{align*}
and 
 \begin{align*}
&8h_{t^1t^2}=-4\mathscr{X}^0\mathscr{X}^2h_{t^1}=-4\mathscr{X}^0\partial_{t^1}\mathscr{X}^2 h=\mathscr{X}^0_\alpha[\mathscr{X}^0,\mathscr{X}^1]\mathscr{X}^2 h\\=&(\mathscr{X}^0)^2\mathscr{X}^1\mathscr{X}^2 h-\mathscr{X}^0\mathscr{X}^1\mathscr{X}^0\mathscr{X}^2 h =-2(\mathscr{X}^0)^2h_{t^3}-4h_{t^1t^2},
        \end{align*}
we have $h_{t^1t^2}=(\mathscr{X}^0)^2h_{t^3}=0$. Analogous calculations show that
\begin{align}\label{4.8}
    h_{t^\nu t^{\mu}}=0, (\mathscr{X}^i)^2h_{t^\nu}=0,
\end{align}
  for $\nu,\mu=1,2,..,7$, $\nu\neq \mu$ and $i=0,1,...,7$.\\
  Furthermore, we have
 \small{\begin{align*}
      &8h_{t^1t^1}=4\mathscr{X}^3\mathscr{X}^2 h_{t^1}=4\mathscr{X}^3\partial_{t^1}\mathscr{X}^2 h=\mathscr{X}^3[\mathscr{X}^5,\mathscr{X}^4]\mathscr{X}^2h\\=&\mathscr{X}^3\mathscr{X}^5\mathscr{X}^4\mathscr{X}^2 h-\mathscr{X}^3\mathscr{X}^4\mathscr{X}^5\mathscr{X}^2h=4h_{t^6t^6}+4h_{t^7t^7}.  \end{align*}}
      Similar calculations yield
    \small{\begin{align*}
2h_{t^1t^1}=h_{t^2t^2}+h_{t^3t^3}=h_{t^4t^4}+h_{t^5t^5}=h_{t^7t^7}+h_{t^6t^6},
      \end{align*}
 \begin{align*}
2h_{t^2t^2}=h_{t^3t^3}+h_{t^1t^1}=h_{t^4t^4}+h_{t^6t^6}=h_{t^5t^5}+h_{t^7t^7},
      \end{align*}
 \begin{align*}
2h_{t^3t^3}=h_{t^1t^1}+h_{t^2t^2}=h_{t^4t^4}+h_{t^7t^7}=h_{t^6t^6}+h_{t^5t^5},
      \end{align*}
 \begin{align*}
          2h_{t^4t^4}=h_{t^5t^5}+h_{t^1t^1}=h_{t^6t^6}+h_{t^2t^2}=h_{t^7t^7}+h_{t^3t^3},
      \end{align*}
 \begin{align*}
          2h_{t^5t^5}=h_{t^1t^1}+h_{t^4t^4}=h_{t^7t^7}+h_{t^2t^2}=h_{t^3t^3}+h_{t^6t^6},
      \end{align*}
 \begin{align*}
          2h_{t^6t^6}=h_{t^1t^1}+h_{t^7t^7}=h_{t^2t^2}+h_{t^4t^4}=h_{t^5t^5}+h_{t^3t^3},
      \end{align*}
 \begin{align*}
          2h_{t^7t^7}=h_{t^6t^6}+h_{t^1t^1}=h_{t^2t^2}+h_{t^5t^5}=h_{t^3t^3}+h_{t^4t^4}.
      \end{align*}}
Then, we conclude that $h_{t^\nu t^\nu}=h_{t^\mu t^\mu}$  and $h_{t^\nu t^\nu t^\nu}=h_{t^\mu t^\mu t^\nu}=0$ for $\nu\neq \mu=1,2,...,7$. 

Next, we prove that $h_{t^\nu t^\nu}$ is constant. Using (\ref{4.8}), we have
\begin{align*}
    0&=\partial_{t^\nu}(I_\nu \mathscr{X}^i) (\mathscr{X}^i)^2 h=\partial_{t^\nu}\mathscr{X}^i(I_\nu \mathscr{X}^i)\mathscr{X}^i h+\partial_{t^\nu}[I_\nu \mathscr{X}^i,\mathscr{X}^i]\mathscr{X}^ih\\&=2\mathscr{X}^i h_{t^\nu t^\nu}+4\mathscr{X}^i h_{t^\nu t^\nu}=6\mathscr{X}^i h_{t^\nu t^\nu},
\end{align*}
for any $\nu=1,2,...,7$ and $i=0,1,...,7$.

 We have proved the vanishing of all derivatives of the of $h_{t^\nu t^\nu}$, which means $h_{t^\nu t^\nu}$ is constant, which we denote by $2\kappa$. Let us point out that $\kappa>0$ since $h>0$ and $g$ is independent of $t^\nu$.

\end{proof}

We define $h=g(x)+\kappa H$, where
\begin{align*}
    H=\Sigma_{\nu=1}^7(t^\nu+s^\nu(x))^2\end{align*}

Therefore, we come to

\begin{lemma}
For  $i,j=0,...,7$, $\nu,\mu=1,2,...,7$, from $\mathcal{D}(\mathscr{X}^i,\mathscr{X}^j)=0$, we have\\
(1)
\begin{align}\label{4.15}
    \mathscr{X}^i \mathscr{X}^i h_{t^\nu}=0,
\end{align}
(2)
for $i\neq j$,
\begin{align}\label{4.17}
    \mathscr{X}^i \mathscr{X}^j h_{t^\nu}= 4\delta_{\nu\mu}\kappa,
\end{align}
where $\mathscr{X}^i \mathscr{X}^j h=2h_{t^\mu}$ for some $\mu$.

\end{lemma}
From above, we have all second derivatives of $s^\nu(x)$ vanish. Thus, $s^\nu(x)$ must be linear function.
Now we compute $s^\nu(x)$, we begin with
\begin{equation}
    \begin{aligned}
        \mathscr{X}^0 h_{t^1}=&\frac{1}{2}\mathscr{X}^0\mathscr{X}^3\mathscr{X}^2 h=-2\mathscr{X}^2 h_{t^3}-\mathscr{X}^3 h_{t^2}\\ 
=&-\frac{1}{2}\mathscr{X}^0\mathscr{X}^2\mathscr{X}^3 h=2\mathscr{X}^3 h_{t^2}+\mathscr{X}^2 h_{t^3}\\
=&\frac{1}{2}\mathscr{X}^0\mathscr{X}^5\mathscr{X}^4 h=-2\mathscr{X}^4 h_{t^5}-\mathscr{X}^5 h_{t^4}\\
=&-\frac{1}{2}\mathscr{X}^0\mathscr{X}^4\mathscr{X}^5 h=2\mathscr{X}^5 h_{t^4}+\mathscr{X}^4 h_{t^5}\\
=&\frac{1}{2}\mathscr{X}^0\mathscr{X}^6\mathscr{X}^7 h=-2\mathscr{X}^7 h_{t^6}-\mathscr{X}^6 h_{t^7}\\
=&-\frac{1}{2}\mathscr{X}^0\mathscr{X}^7\mathscr{X}^6 h=2\mathscr{X}^6 h_{t^7}+\mathscr{X}^7 h_{t^7}.\\  
    \end{aligned}\end{equation}
   That is 
   \small{\begin{align}
        \mathscr{X}^0 s^1=-\mathscr{X}^2 s^3=\mathscr{X}^3 s^2=-\mathscr{X}^4 s^5=\mathscr{X}^5 s^4=-\mathscr{X}^7s^6=\mathscr{X}^6 s^7.
    \end{align}}
    In the same way, we have
   \small{ \begin{align*}
         &\mathscr{X}^1 s^1=\mathscr{X}^2 s^2=\mathscr{X}^3 s^3=\mathscr{X}^4s^4=\mathscr{X}^5s^5=\mathscr{X}^6 s^6=\mathscr{X}^7 s^7,\\& \mathscr{X}^2 s^1=\mathscr{X}^0s^3=-\mathscr{X}^1 s^2=-\mathscr{X}^4 s^7=\mathscr{X}^5 s^6=\mathscr{X}^7 s^4=-\mathscr{X}^6 s^5,\\&
          \mathscr{X}^3 s^1=-\mathscr{X}^0 s^2=-\mathscr{X}^1 s^3=\mathscr{X}^4 s^6=\mathscr{X}^5 s^7=-\mathscr{X}^7s^5=-\mathscr{X}^6s^4, \\&          
          \mathscr{X}^4 s^1=\mathscr{X}^0 s^5=-\mathscr{X}^1 s^4=   \mathscr{X}^2 s^7=-\mathscr{X}^3 s^6=-\mathscr{X}^7 s^2=\mathscr{X}^6s^3,\\& \mathscr{X}^5 s^1=-\mathscr{X}^0 s^4=-\mathscr{X}^1 s^5=-\mathscr{X}^2 s^6=-\mathscr{X}^3 s^7=\mathscr{X}^7 s^3=\mathscr{X}^6 s^2, \\& \mathscr{X}^6s^1=-\mathscr{X}^0 s^7=-\mathscr{X}^1 s^6=\mathscr{X}^2 s^5=\mathscr{X}^3 s^4=-\mathscr{X}^4 s^3=-\mathscr{X}^5 s^2 , \\& \mathscr{X}^7 s^1=\mathscr{X}^0s^6=-\mathscr{X}^1 s^7=-\mathscr{X}^2 s^4=\mathscr{X}^3 s^5=\mathscr{X}^4 s^2=-\mathscr{X}^5 s^3.\end{align*}}
Then we conclude that there are $y\in \mathbb{R}^{8}$ and $c\in \mathbb{R}^7$,
\small{
\begin{align*}
    s^1(x)=-x^1 y^0+x^0 y^1-x^3y^2+x^2 y^3-x^5 y^4+x^4 y^5+x^7 y^6 -x^6 y^7+c_1,
\end{align*}
\begin{align*}
    s^2(x)=-x^2 y^0+x^3 y^1+x^0y^2-x^1 y^3-x^6 y^4-x^7 y^5+x^4 y^6 +x^5 y^7+c_2,
\end{align*}
\begin{align*}
    s^3(x)=-x^3y^0-x^2 y^1+x^1y^2+x^0 y^3-x^7y^4+x^6 y^5-x^5y^6+x^4 y^7+c_3,
\end{align*}
\begin{align*}
    s^4(x)=-x^4 y^0+x^5 y^1+x^6y^2+x^7 y^3+x^0y^4-x^1y^5-x^2 y^6-x^3 y^7+c_4,
\end{align*}
\begin{align*}
    s^5(x)=-x^5 y^0-x^4 y^1+x^7y^2-x^6 y^3+x^1 y^4+x^0 y^5+x^3 y^6-x^2 y^7+c_5,
\end{align*}
\begin{align*}
    s^6(x)=-x^6y^0-x^7 y^1-x^4y^2+x^5 y^3+x^2 y^4-x^3y^5+x^0y^6+x^1y^7+c_6,
\end{align*}
\begin{align*}
    s^7(x)=-x^7 y^0+x^6 y^1-x^5y^2-x^4 y^3+x^5y^4+x^2 y^5-x^1 y^6 +x^0 y^7+c_7.
\end{align*}}

The sub-Laplace operator enjoys translation invariance on the 15 dimensional octonionic Heisenberg group. After a suitable translation, we have
\begin{align}
    h(x,t)=g(x)+\kappa H=g(x)+\kappa \Sigma_{\nu=1}^7 (t^\nu)^2.
\end{align}

Notice that all fifth-order horizontal derivatives of $H$ vanish. In particular, the fifth-order derivatives of $h$ and $g(x)$ coincide. To compute $g(x)$, we need the following lemma.

\begin{lemma}\label{lemma 4.5}
    For $i=0,1,...,7$, $\mathcal{D}(\mathscr{X}^i,\mathscr{X}^i)=\mathcal{E}(\mathscr{X}^i,\mathscr{X}^i)=0$, implies
    \begin{align}\label{4.14}
        8\mathscr{X}^i\mathscr{X}^ih-6h^{-1}\bigg[\mathscr{X}^ih \mathscr{X}^ih+\Sigma_{\nu=1}^7(I_\nu \mathscr{X}^i)(I_\nu \mathscr{X}^i) h\bigg]=32 
    \end{align}

\end{lemma}
\begin{proof}
    $\mathcal{D}(\mathscr{X}^i,\mathscr{X}^i)+\mathcal{E}(\mathscr{X}^i,\mathscr{X}^i)$ gives $(\ref{4.14})$.
         
\end{proof}

From (\ref{4.14}), we have
\begin{align}
    8(\mathscr{X}^i)^2h-6h^{-1}\Sigma_{j=0}^7(\mathscr{X}^j h)^2     =32.
\end{align}
Then we have
\begin{align}
    (\mathscr{X}^0)^2 h=(\mathscr{X}^1)^2 h=...=(\mathscr{X}^7)^2 h.
\end{align}
Recall that $\mathscr{X}^1\mathscr{X}^0 h=2h_{t^1}=4\kappa t^1$, we have
\begin{align}
   ( \mathscr{X}^0)^3h=\mathscr{X}^0 (\mathscr{X}^1)^2 h=\mathscr{X}^1\mathscr{X}^0\mathscr{X}^1 h+[\mathscr{X}^0,\mathscr{X}^1]\mathscr{X}^1 h=-6\mathscr{X}^1 h_{t^1}=24\kappa x^0.
\end{align}
Thus, $(\mathscr{X}^0)^4h=24\kappa$. In the same way, for $i=0,1,...,7$, we have
\begin{align}
    (\mathscr{X}^i)^4 h=24\kappa,\, (\mathscr{X}^i)^3h=24\kappa x^i,
\end{align}
After substituting $H$, we can compute 
\begin{align}
    (\mathscr{X}^i)^3g(x)=24\kappa x^i,\,(\mathscr{X}^i)^4g(x)=24\kappa.
\end{align}
\begin{align}
    (\mathscr{X}^i)^2g(x)=(\mathscr{X}^j)^2g(x)\,\,\textrm{for}\, i\neq j.
\end{align}
Then $g(x)$ is a polynomial of degree $4$, and of the form
\begin{align}\label{4.20}
    g(x)=\kappa \bigg[\Sigma_{i=0}^7(x^i)^4+2\Sigma_{i,j=0,i\neq j}^7 (x^i x^j)^2\bigg]+\varepsilon\Sigma_{i=0}^7(x^i)^2+p_1(x),
\end{align}
where $p_1$ is a polynomial of degree $1$ and $\varepsilon$ is a constant. To prove $p_1(x)$ is a constant, first, we use (\ref{3.13}) and (\ref{3.14}) together with $D(\mathscr{X}^i)=B_1(\mathscr{X}^i)=0$ to get for $i=0,1,...,7$,
\begin{align*}
    2\mathscr{X}^if=-4 I_{1}\mathscr{X}^ih_{t^1}=16\kappa x^i,
    \end{align*}
then from (\ref{ft}), we have
\begin{align*}
    f=4\kappa|x|^2+c,
\end{align*}
for some constant $c\geq 4$ since $f=4+\frac{|\nabla_G h|^2}{4h}\geq 4$. Besides,
\begin{align*}
    0=h^2B_1(\mathscr{X}^0)&=2h\mathscr{X}^1h_{t^1}-\Sigma_{\mu=1}^7h_{t^\mu}\mathscr{X}^\mu h+\frac{1}{2} f\mathscr{X}^0 h\\&=-8\kappa[g(x)+\kappa\Sigma_{\nu=1}^7 (t^\nu)^2]x^0-2\kappa \Sigma_{\mu=1}^7t^\mu \mathscr{X}^\mu h+\frac{1}{2}(c+4\kappa|x|^2)\mathscr{X}^0 h.
\end{align*}
Comparing the coefficient of constant, we have $\mathscr{X}^0p_1(x)=0$. The coefficients of $t^\nu$ give $\mathscr{X}^\nu p_1(x)=0$. Then $p_1(x)$ is constant and $g(x)$ is symmetric with respect to $|x|$. Plugging (\ref{4.20}) into (\ref{3.1}), we get our conclusion. Note that the final conclusion can also be reached by \cite{GV2001} since we have proved that $h$ is cylindrically symmetric. \\

\end{document}